\documentclass[reqno]{amsart}
\usepackage{amssymb, amsmath, amsthm}
\usepackage[letterpaper, margin=1in]{geometry}
\usepackage{enumitem}
\usepackage{esint, nicefrac}

\newtheorem{theorem}{Theorem}[section]
\newtheorem{lemma}[theorem]{Lemma}
\newtheorem{proposition}[theorem]{Proposition}

\theoremstyle{definition}
\newtheorem{definition}{definition}[section]
\theoremstyle{remark}
\newtheorem{remark}[definition]{Remark}

\numberwithin{equation}{section}

\newcommand{\abs}[1]{\lvert#1\rvert}

\newcommand{\norm}[1]{\left\lVert#1\right\rVert}
\newcommand{\R}{\mathbb{R}}

\newcommand{\N}{\mathbb{N}}
\newcommand{\Z}{\mathbb{Z}}

\title{Nonunique tangent maps at isolated singularities of minimizing $p$-harmonic maps}
\author{Jonas Hirsch}

\begin{document}
\begin{abstract}
The analysis of ``tangent maps'' at singular points of energy minimizing maps plays an important role in our understanding of the fine structure of the singular set.
This note presents the first example of a minimizing (not just stationary) $p$-harmonic map with nonunique tangent maps at an isolated singularity. We construct a $n$-dimensional manifold $N$ such that for every admissible tuple $p< m\le n+2$, there exists a map from $B_1^m$ into $N$ that minimizes the $p$-energy, has an isolated singularity at the origin and admits a continuum of distinct tangent maps. The construction builds upon and extends B.~ White's example for $p=2$ in the stationary case \cite{White}.
\end{abstract}
\maketitle
\vspace{-1.5em}
\begin{center}
\textit{A tribute to Brian White}
\end{center}

\section*{Introduction}
Let $\Omega \subset \R^m$ be a bounded domain and $N$ a compact smooth Riemannian manifold of dimension $n\ge 2$ isometrically embedded in some Euclidean space $\R^l$. The $p$-energy of a map $u\in W^{1,p}(\Omega,N):=\{ v \in W^{1,p}(\Omega, \R^l) \colon v(x) \in N \text{ a.e. }\}$ is 
\[ E_p(u):= \int_{\Omega} \abs{Du}^p. \]
The map $u$ is called minimizing in $\Omega$ iff for all $v \in W^{1,p}(\Omega,N)$ with $u-v \in W^{1,p}_0(\Omega, \R^l)$ one has $E_p(u) \le E_p(v)$.\\

The direct method of the calculus of variations provides, for any given $v_0 \in W^{1,p}(\Omega, \R^l)$ the existence of at least one energy minimizing map $u$ with the same boundary data.  

A map $u$ locally minimizing the $p$-energy is of class $C^{1,\alpha}$ outside of a set of Haussdorff dimension at most $m-[p]-1$. This was shown for $p=2$ by Schoen and Uhlenbeck, in \cite{SchoenUhlenbeck}, it had been extended to general $p$ by F.~Fuchs \cite{Fuchs} , R.~Hardt, F.~Lin \cite{HardtLin}, and S.~Luckhaus \cite{Luckhaus} extended their result to general $p$.

Appealing to the monotonicity formula and the sub-sequential compactness of harmonic maps, they showed that for any $y \in \Omega$ and every sequence $r_i \searrow 0$, a subsequence of the maps
\[ x \mapsto u(y+ r_ix) \]
converges locally in energy to a map $u_\infty\colon \R^m \to N$. Such a map $u_\infty$ is zero homogeneous, itself locally minimizing the $p$-energy and called a \emph{tangent map} of $u$ at $y$.

Hence there is a special interest in understanding this particular class of energy minimizer, more precisely the set of possible tangent maps at a given point $y$. 
For instance one could hope that a tangent map $u_\infty$ provides a good picture of $u$ near $y$. This would not be true if there were more than one tangent map at $y$ meaning there are two different subsequences that converge to different tangent maps. 

Whether such pathological behavior is possible remains unclear to the best of the author's knowledge, even though Brian White announced the existence of such an example in \cite{White}.

The uniqueness of tangent map has been shown
\begin{enumerate}
	\item if there is one constant tangent map at $y$ because then $y$ is actually a regular point
	\item if the target manifold is analytic and singularity $y$ is isolated, proven in the groundbreaking work of L.~Simon, \cite{Simon1983}
	\item if $m=3$ and $N$ is two-dimensional in which case the Gulliver and White could prove a rate of convergence, \cite{GulliverWhite1989}
	\item and more recently with the help an epiperimetric inequality once again for analytic targets, \cite{CaniatoParise2025}.
\end{enumerate}
Results (2) through (4) have been proven in the context of classical $2$-harmonic maps. However, only (1) holds for general $p$. However, the author expects that L.~Simon's result probably extends to general $p$.\medskip

The aim of this note is to give an example for the pathological behaviour:

\begin{theorem}\label{thm:white*}
For every $p>1$, there is a constant $m(p)$ such that for any $m>m(p)$ there exists a smooth manifold $N$ of dimension $n\ge m+2$ and a $p$-harmonic map $U \colon B^m_1 \to N$ with an isolated singularity at $0$ and a continuum of distinct tangent maps. Furthermore $U$ is the unique minimising map with these boundary data.
\end{theorem}

The argument is inspired by the example of B.~White for a stationary but not necessarily minimizing harmonic map with a continuum of distinct tangent maps. 

It is worth mentioning that the manifold $N$ is kind of universal that for every compact set $\mathcal{P}$ of parameters $(p,m)$ one can choose $N$ independent of the individual elements in $\mathcal{P}$.
 
\section*{Acknowledgment}
I want to thank B.~White for his inspiring works. Furthermore I am grateful to Katarzyna Mazowiecka for bringing the article \cite{Hong2000} to my attention. Finally I would like to thank the DFG for supporting my research.

\section{The ``universal'' manifold}
Following B.~White we consider the ``universal'' $n+2$ dimensional manifold $N$ being the product $\left(T^1\times \R \right)\times S^n$ associated with the metric
\begin{equation}\label{eq.the metric universal manifold}
	g=d\theta^2+dy^2+(B_0-V(\theta,y))\,g_0\,;
\end{equation}
where $(T^1=\R / \Z, d\theta^2)$ is the flat torus, $g_0$ is the round metric on $S^n$, the large constant $B_0$ satisfies at least \[B_0> \sup (|V|+ |\nabla V|)+1\,.\] 

Even though most of our analysis does not depend on the specific choice of the potential $V$, we take it to be the very same as B.~White proposed i.e. 

\begin{equation}\label{eq.potential}
	V(\theta,y)=-e^{-\nicefrac{1}{y^2}}\sin\left(\theta+\nicefrac1y\right)
\end{equation}

This defines a complete metric on $N$ for any $n \in \N$. We will always assume that $n \ge m-1$ but otherwise the result does not depend on the specific choice of $n$. In that sense the manifold is ``universal''.

Each map $U$ taking values in $N$ can be divided into the part $u$ taking values in $T^1\times R$ and $v$ taking values in $S^n$. 

The associated $p$-energy of a map $U=(u,v) \colon B_1^m \to N$ is henceforth
\[E(U)= \int_{B_1}  \left(|Du|^2 + B(u) \,|Dv|^2\right)^{\frac{p}{2}} \,.\]

Independent of the dimension of the spherical part $S^n$ in $N$ one has that the orthogonal group $O(m)$ acts on $B_1\subset \R^m$ and $N$. In the latter case by $g\cdot (\theta,y,z)= (\theta,y,g\,z)$ where we consider $g\in O(m) \subset O(n+1)$ by identifying $\R^m$ with $\R^m\times \{0\} \subset \R^n$. 
We will be particular interested in maps that are $O(m)$-equivariant i.e. $U(g\,x)=g\cdot U(x)$. One realises that every equivariant map is of the form 
\begin{equation}\label{eq.equivariant maps}
	U(x)=\bigl(u(|x|), \pm \frac{x}{|x|}\bigr)
\end{equation}
where $u\colon (0,1) \to T^1 \times \R$ is radial symmetric. \\

The theorem is a corollary of the following two propositions, that will be proven in the next two sections.

\begin{proposition}\label{prop.equivariant minimizers}
	There are constants $\epsilon_1=\epsilon_1(m,p), \,\overline{B}_0=\overline{B}_0(m,p,\epsilon_1,|V|,|\nabla V|)$ such that any energy minimizing map $U=(u, v) \colon B_1^m \to N$ with
	\begin{enumerate}
		\item\label{a.equivariant boundary} equivariant boundary data, i.e. $U(x)=(u_0, x)$ on $\partial B_1$, $u_0 \in T^1\times \R$ constant. 
		\item\label{a.small energy} $E(U)\le \epsilon_1 \, B_0^{\nicefrac{(p+1)}{2}}$ 
		\item\label{a.admissible parameters} $4(m-1)<(m-p)^2$ and $B_0>\overline{B}_0$ 
	\end{enumerate}
	 is equivariant.
\end{proposition}

A consequence of the second proposition will be that the equivariant harmonic maps we are interested in will satisfy the assumption \eqref{a.small energy} on the energy.

\begin{remark}
	The condition $4(m-1)<(m-p)^2$ is necessary in the sense that if $4(m-1)>(m-p)^2$ the hedge hog $v_0(x)=\frac{x}{|x|}$ is not even a stable harmonic map from $B_1$ into $S^m$, compare \cite{Hong2000}. 
\end{remark}

\begin{proposition}\label{prop.equivariant stationary}
There is an equivariant $p$-harmonic map $U(x)=(u(|x|),\frac{x}{|x|})$ that is the unique minimizer with respect to its own boundary data and additionally satisfies 
	\begin{enumerate}
		\item\label{c.small energy} $E(U) \le C B_0^{\nicefrac{p}{2}}$,
		\item\label{c.tangent map}
		 $0$ is an isolated singularity and its tangent maps there are the continuum 
		 \[\left\{ (\theta,0,\frac{x}{|x|}) \colon \theta \in T^1\right\}\,.\]
	\end{enumerate}
\end{proposition}
 
\section{Proof of Proposition \ref{prop.equivariant minimizers}}\label{sec.equivariant minimizers}
We will divide the proof into the following four substeps: 
\begin{itemize}
	\item[\emph{Step 1:}] Let $\tilde{U}=(\tilde{u},v)\colon B_1 \to N$ be given such that $\tilde{u}=u_0$ constant on $\partial B_1$ for some $u_0 \in T^1 \times \R$ then there exists $u(x)=u(|x|)$ radial symmetric with $u(1)=u_0$ such that 
	\[ E((u,v)) \le E(\tilde{U})\,. \]
	Equality implies that $\tilde{u}$ is radial. 
	\item[\emph{Step 2:}] Let $U=(u(|x|), v) $ be a stationary $p$-harmonic map then 
	\begin{equation}\label{eq.bound of Step 2}
		B_0|r\,u'(r)|^{p-1}\le C \norm{\nabla V}_\infty E(U)\,.
	\end{equation}
	\item[\emph{Step 3:}] Let $U=(u(|x|),v)$ be energy minimizing satisfying $v(x)=x$ on $\partial B_1$ and the energy bound \eqref{a.small energy} then $v(x)=\frac{x}{|x|}$ i.e. $U$ is equivariant. 
\end{itemize}

\subsubsection*{to Step 1:}\label{subsubsec.toStep1} Let $\tilde{u}_g(x)=\tilde{u}(g\,x)$ for any $g \in SO(m)$. Since $\tilde{u}=u_0$ is constant on $\partial B_1$ the map $\tilde{U}_g=(\tilde{u}_g, v)$ is a competitor to $\tilde{U}$ with $E(\tilde{U}_g)=E(\tilde{U})$ thus we may integrate over the Haar measure $\mu$ on $SO(m)$ and obtain 
\begin{align*}
	E(\tilde{U}) = \fint_{SO(m)} E(\tilde{U}_g) \, d\mu(g) = \fint_{SO(n)} \int_{B_1}  \left(|D\tilde{u}(gx)|^2 + B(\tilde{u}(gx)) \,|Dv(x)|^2\right)^{\nicefrac{p}{2}} \, dxd\mu(g)\,.
\end{align*}
Using that
\[\fint_{SO(m)} f(gx)\, d\mu(g) = \fint_{\partial B_{|x|}} f(y)\, ds(y)\]
we can rewrite the energy 
\begin{align}\nonumber
	\frac{E(\tilde{U})}{|\partial B_1|}&= \fint_{\partial B_1}\fint_{\partial B_1} \int_0^1 \left( |D\tilde{u}(ry)|^2 + B(\tilde{u}(ry)\,|Dv(rz)|^2\right)^{\nicefrac p2}\, r^{n-1}drds(y)ds(z)\\
	&\ge  \fint_{\partial B_1}\fint_{\partial B_1} \int_0^1 \left( |\frac{\partial}{\partial r}\tilde{u}(ry)|^2 + B(\tilde{u}(ry)\,|Dv(rz)|^2\right)^{\nicefrac p2}\, r^{n-1}drds(y)ds(z)\label{eq.only radial}
\end{align}
Thus there must be $y_0\in \partial B_1$ such that 
\begin{align*}
	E(\tilde{U})&\ge \int_{\partial B_1} \int_0^1 \left( |D\tilde{u}(ry_0)|^2 + B(\tilde{u}(ry_0)\,|Dv(rz)|^2\right)^{\nicefrac p2}\, r^{n-1}drds(y)ds(z) \\
	&\ge \int_{\partial B_1} \int_0^1 \left( |D\tilde{u}(ry_0)|^2 + B(\tilde{u}(ry_0)\,|Dv(rz)|^2\right)^{\nicefrac p2}\, r^{n-1}drds(y)ds(z)\\
	&=E(U)\,
\end{align*}
where $U=(\tilde{u}(|x|y_0), v(x))$.
Note that in case of equality \eqref{eq.only radial} implies that $|D\tilde{u}(x)|=|\frac{\partial}{\partial_r} \tilde{u}(x)|$ for a.e. $x$ or equivalently that the tangential energy $|D_\tau \tilde{u}(x)|=0$ a.e. This implies that $u$ is radial symmetric, i.e. $\tilde{u}(x)=\tilde{u}(|x|)$.

\subsubsection*{to Step 2:}\label{subsubsec.toStep2}
The monotonicity formula\footnote{It can obtained by testing the inner variation with the Lipschitz continuous vectorfield $X(x)= \frac{x-y}{\max\{|x-y|,r\}^{m-p}}- \frac{x-y}{R^{m-p}}$, that is vanishing on $\partial B_R(y)$ and therefore admissible.} for stationary $p$-harmonic maps states that for any ball $B_R(y)\subset \Omega, 0<r<R$ and the ``$y$-radial derivative'' $\frac{\partial}{\partial r_y}U(x) = DU(x)\frac{(x-y)}{|x-y|}$  
\begin{equation}\label{eq.monotonicityformula}
	R^{p-m} \int_{B_R(y)} |DU|_g^p - r^{p-m} \int_{B_r(y)} |DU|_g^p = p \int_{B_R(y)\setminus B_r(y)} \frac{|DU|_g^{p-2} |\frac{\partial}{\partial r_y}U|_g^2}{|x-y|^{m-p}}\,.
\end{equation}

In particular for $y=0$ we find by H\"older's inequality and a subsequent application of the monotonicity formula 
\begin{align*}
	\int_{B_r} |DU|_g^{p-2}|u'| &\le \left( \int_{B_r}|DU|_g^{p-2} |u'|^2\right)^{\nicefrac12} \left( \int_{B_r} |DU|_g^p \right)^{\nicefrac{(p-2)}{2p}} |B_r|^{\frac1p}\\
	&\lesssim  r^{(m-p)(1-\nicefrac{1}{p})} E(U)^{1-\nicefrac1p} r^{\nicefrac mp} \lesssim E(U)^{1-\nicefrac1p} r^{m+1-p}\,.
\end{align*}
Therefore we can find a sequence $r_k\downarrow 0 $ s.t. for all $k$ \[\int_{\partial B_{r_k}} |DU|_g^{p-2} |u'| \le \frac{8}{r_k} \int_{B_{2r_k}}|DU|_g^{p-2}|u'| \lesssim 8 E(U)^{1-\nicefrac 1p}r_k^{m-p}\,.\]
It converges to $0$ a $k \to \infty$ since $m>p$. \\
For any radially symmetric first component, $u=u(|x|)$, which we consider as a function that depends on only one variable, the energy can be written as $E(U)= \int_0^1 L(r,u,u')\, dr $, where $L(r,u,u')=\int_{\partial B_r} |DU|_g^p \, ds$. As a stationary point of this one-dimensional Lagrangian, it must satisfy 
\begin{equation}\label{eq.lagrangian - EUL}
	\frac{d}{dr} \frac{\partial L}{\partial u'}=\frac{\partial L}{\partial u}\,.
\end{equation}
One readily checks
\begin{align*}
	\frac{\partial L}{\partial u'} &= p\int_{\partial B_r} |DU|_g^{p-2}u'\,ds\\
	\frac{\partial L}{\partial u} &= p \int_{\partial B_r} |DU|^{p-2}_g |Dv|^2\,\frac{\partial B}{\partial u}\, ds\,.
\end{align*}

Hence we may integrate \eqref{eq.lagrangian - EUL} on $r \in (r_k, \rho)$ to obtain
\[ \left(\int_{\partial B_\rho} |DU|_g^{p-2}u'\,ds\right) - \left(\int_{\partial B_{r_k}} |DU|_g^{p-2}u'\,ds\right) = \int_{B_\rho \setminus B_{r_k}} |DU|^{p-2}_g |Dv|^2\,\frac{\partial B}{\partial u}\,.\]
Due to our choice of $r_k$ we can take the limit $r_k\downarrow 0$ to obtain \eqref{eq.bound of Step 2}
\begin{align}\nonumber
	|\partial B_\rho| |u'(\rho)|^{p-1} &\le \sup\frac{|\frac{\partial B}{\partial u}|}{B} \int_{B_\rho} |DU|_g^{p-2} B |Dv|^2 \le \sup\frac{|\frac{\partial B}{\partial u}|}{B} E(U) \rho^{m-p}\\
	&\le C \,(\epsilon_1 B_0)^{\nicefrac{(p-1)}{2}} \rho^{m-p} \label{eq.first gradient bound}
\end{align}

\subsubsection*{to Step 3:}\label{subsubsec.toStep3}
We take inspiration of the argument used in \cite[section 2.1]{Hong2000}. His argument relies on the following classical version of Hardys inequality: 
\[ (m-p)^2 \int \frac{1}{|x|^p} |\tilde{w}|^2\le 4 \int \frac{1}{|x|^{p-2}}\left|\frac{\partial \tilde{w}}{\partial r}\right|^2 \,,\]
for any $\tilde{w} \in H^1_0(\R^m) \cap L^\infty$ with equality if and only if $\tilde{w}\equiv 0$. 

The equivariant map $v_0(x)=\frac{x}{|x|}$ is stationary for all $m\ge p$ i.e.
\begin{equation*}\label{eq.hedge hog is p-stationary}
\int |Dv_0|^{p-2} Dv_0\cdot D	w = \int |Dv_0|^p\, v_0\cdot w \quad \forall w \in C_c^\infty(\R^m)
\end{equation*}
Since $v_0$ is $0$-homogeneous i.e. $\partial_r v_0=0$ and $|Dv_0|^2 = \frac{m-1}{|x|^2}$ we deduce for any measurable $a(r) \in L^\infty(B_1)$ by approximation, that
 \begin{equation}\label{eq.hedge hog is p-stationary-v2}
\int \frac{a(r)}{|x|^{p-2}} Dv_0\cdot D	w = \int \frac{a(r)}{|x|^{p-2}}|Dv_0|^2 \;v_0\cdot w \quad \forall w \in W^{1,p}_0(B_1) 
\end{equation}
In case $w=v-v_0$, where $v\in W^{1,p}(B_1, S^n)$ is a given competitor to $v_0$, we have $ -2 v_0\dot w = |w|^2$. Multiplying the above identity by $-2$ and applying Hardy's inequality gives 
\begin{align*}
	-2\int \frac{a(r)}{|x|^{p-2}} Dv_0\cdot D	w = \int \frac{a(r)}{|x|^{p-2}}|Dv_0|^2 |w|^2 \le \frac{4(m-1)}{(m-p)^2}\, \frac{\max a}{\min a}\int \frac{a(r)}{|x|^{p-2}}|Dw|^2\,.
\end{align*}
We found that, by rearranging the above and abbreviating $\delta = \frac{4(m-1)}{(m-p)^2}\, \frac{\max a}{\min a}$, that
\begin{align}\nonumber
	0 &\le (\delta -1) \int \frac{a(r)}{|x|^{p-2}}|Dw|^2 + \int_{B_1} \frac{a(r)}{|x|^{p-2}}\left(|Dv|^2-|Dv_0|^2
	\right)\\\nonumber
	&\le (\delta -1) \int \frac{a(r)}{|x|^{p-2}}|Dw|^2 + \int_{B_1} \frac{a(r)}{B(u)|x|^{p-2}}\left((|Du|^2 + B(u)|Dv|^2)-(|Du|^2 + B(u)|Dv_0|^2)\,.
	\right)
\end{align}
This inequality suggest the following choice 
\[a(r)=\left( r^2|Du|^2 + (m-1)B(u) \right)^{\nicefrac{(p-2)}{2}} B(u)\,.\]
In step 2 we established \eqref{eq.first gradient bound} i.e. $|ru'(r)|^2\le C \epsilon_1 B_0$. We can estimate by setting $\epsilon_2 = \frac{\sup|V|}{B_0}$ 
\[ B_0^p \left((m-1)(1-\epsilon_2 )\right)^{\nicefrac{(p-2)}{2}}(1-\epsilon_2) \le a(r) \le B_0^p \left(C\epsilon_1 + (m-1)(1+\epsilon_2)\right)^{\nicefrac{(p-2)}{2}}(1+\epsilon_2)\,.\]
Therefore, for sufficiently small $\epsilon_1,\epsilon_2>0$, we deduce that $\theta<1$. Thus, based on our choice of $a(r)$, we conclude that 
\[ \int_{B_1} \left(|Du|^2+ B(u)|Dv_0|^2\right)^{\nicefrac p2} \le \int_{B_1}\left(|Du|^2+ B(u)|Dv_0|^2\right)^{\nicefrac{(p-2)}{2}}\left(|Du|^2+ B(u)|Dv|^2\right)\,.\]
In a first step, H\"older's inequality implies that the function $U=(u,v_0)$ is minimizing. Then, equality in the above argument implies that $|Dw|=0$, so $U=(u,v_0)$.


\section{Proof of Proposition \ref{prop.equivariant stationary}}\label{sec.equivariant stationary}
As before we divide the proof in several steps and two lemmas. The lemmas are in reminiscence to the original article of B.~White. 

\begin{itemize}
	\item[\emph{Step 0:}] Change of coordinates and ``new'' energy with parameters $\beta>\frac12, \alpha>0$
	\begin{equation}\label{eq.new energy}
		E(u)= \int_{-\infty}^R \left(|u'|^2 +(m-1)B(u)\right)^\beta \, e^{\alpha t} dt= \int_\infty^R \ell(u)^\beta \, e^{\alpha t}dt
	\end{equation}
	\item[\emph{Step 1:}] Euler-Lagrange equation, see \eqref{eq.Euler-Lagrange}, and first integral/ ``Hamiltonian'' 
	\begin{equation}\label{eq.hamiltonian}
		H(u)= \ell(u)^{\beta-1}\left( (2\beta -1) |u'|^2 - B(u)\right)
	\end{equation}
	\item[\emph{Step 2:}] the following equivalence holds for critical points of \eqref{eq.new energy}: 
	$u$ has finite energy $\Leftrightarrow$ $H(u)<0$. 
	\item[\emph{Step 3:}] every finite energy critical point is minimizing 
	\item[\emph{Step 4:}] every minimizer of \eqref{eq.new energy} is uniquely minimizing
	\item[\emph{Step 5:}]
	 \begin{lemma}\label{lem.H constant} Let $u$ be a solution to \eqref{eq.Euler-Lagrange} 
	 \begin{enumerate}
	 	\item The following are equivalent: 
	 		\begin{enumerate}
	 		\item\label{eq.H constant: equivalence a} $H$ is constant,
	 		\item\label{eq.H constant: equivalence b} $u$ is constant and $u \in T^{1}\times \{0\} \subset \{ V=0\}\cap \{H=-B_0^\beta\}$
	 		\end{enumerate}
	 	\item If $H(u(0))< - B_0^{\beta}$ then \begin{enumerate} \item\label{eq.unbounded} $\limsup_{t\to \infty} |u(t)|= + \infty$ \item\label{eq.no intersection} $u(t) \notin \{V=0\}$ for all $t>0$.  \end{enumerate}
	 \end{enumerate}
	\end{lemma}
	One should compare this lemma to \cite[Lemma, page 127]{White}
	\item[\emph{Step 6:}] 
	\begin{lemma}\label{lem.non-unique tangent}
		There exists a global solution $u$ to \eqref{eq.Euler-Lagrange} with 
		\begin{enumerate}
			\item\label{lem.non-unique tangent conclusion1} $H(u(t))< - B_0^\beta$ for all $t$ and $\lim_{t \downarrow -\infty} H(u(t)) = - B_0^\beta$;
			\item\label{lem.non-unique tangent conclusion2} $E(u)\lesssim B_0^\beta $
			\item\label{lem.non-unique tangent conclusion3} $T^1\times \{0\}= \{ \lim_{k \to \infty} v(t_k) \colon t_k \downarrow -\infty \}$. 
		\end{enumerate}
	\end{lemma}
\end{itemize}

\subsubsection*{to Step 0:}\label{subsubsec.toStep0-equivariant}
Recall the structure of equivariant maps i.e. \eqref{eq.equivariant maps}. Hence we may choose coordinates $u(e^t)=v(t)$ i.e. $u(|x|)=v(\ln|x|)$ and $u'(|x|)= \frac{v'(|x|)}{|x|}$ hence $(|Du|^2 + B(u)|D\frac{x}{|x|}|^2) = |x|^{-2} \left( |v'|+ (m-1) B(v) \right)$, where $v,v'$ are evaluated at $ln|x|$. Using polar coordinates $x=e^{t}y$ the energy becomes
\[ E(U)= \int_{B_1} \bigl(|Du|^2+ B(u) |D\frac{x}{|x|}|^2\bigr)^{\nicefrac p2} = |\partial B_1| \int_{-\infty}^0 (|v'|^2 + (m-1)B(v))^{\nicefrac p2} \, e^{(m-p)t}dt\,, \] 
i.e. the $p$-harmonic map energy for equivariant maps agrees with \eqref{eq.new energy} for parameters $\alpha=m-p$ and $\beta = \nicefrac{p}{2}$.

\subsubsection*{to Step 1:}\label{subsubsec.toStep1-equivariant} Since \eqref{eq.new energy} is just the action functional to the Lagrangian $\ell(u)^\beta e^{\alpha t}$ the associated Euler-Lagrange equation is just 
\begin{align}\label{eq.Euler-Lagrange}\nonumber
	&\frac{d}{dt} \left(e^{\alpha t}\,\ell(u)^{\beta -1} \frac{\partial \ell}{\partial u'}\right) =  \left(e^{\alpha t} \,\ell(u)^{\beta -1} \frac{\partial \ell}{\partial u}\right)\,\\
	\text{ or }\quad 2&\frac{d}{dt} \left( e^{\alpha t} \,\ell(u)^{\beta -1}  u'\right) =  - \left(e^{\alpha t} \ell(u)^{\beta -1} \frac{\partial V}{\partial u} \right)\\\nonumber
	\text{ or }\quad \;\;&\left(I+ 2(\beta-1) \frac{u'\otimes u'}{\ell(u)}\right) u'' +\alpha u' = - \frac12 \frac{\partial V}{\partial u}.
\end{align}
Since  $\min\{1,2\beta -1\} I \le \left(I+ 2(\beta-1) \frac{u'\otimes u'}{\ell(u)}\right) \le \max\{1,2\beta-1\} I$ the matrix is for $\beta >\frac12$ uniformly invertible and since $\sup|DV|<\infty$ solutions to  \eqref{eq.Euler-Lagrange} exists for all times. \medskip

The interpretation as a mechanical system immediately suggest to consider the associated energy i.e. the first integral associated to a solution $u$ of \eqref{eq.Euler-Lagrange} 
\begin{align*}
	-\alpha \ell(u)^\beta e^{\alpha t}=- \frac{\partial (\ell(u)^\beta e^{\alpha t})}{\partial t}= \frac{d}{dt} \left(\left(\beta \ell(u)^{\beta -1} \frac{\partial \ell}{\partial u'}\cdot u'- \ell(u)^\beta\right)e^{\alpha t} \right)\\=e^{\alpha t} \frac{d}{dt} \left(\beta \ell(u)^{\beta -1} \frac{\partial \ell}{\partial u'}\cdot u'- \ell(u)^\beta\right) + \alpha e^{\alpha t} \left(\beta \ell(u)^{\beta -1} \frac{\partial \ell}{\partial u'}\cdot u'- \ell(u)^\beta\right)
\end{align*}
Hence we deduce that for the introduced Hamiltonian \eqref{eq.hamiltonian} along a particle $u$
\begin{align}\label{eq.derivative of H}
	\frac{d}{dt} H(u) = \frac{d}{dt} \left(\beta \ell(u)^{\beta -1} \frac{\partial \ell}{\partial u'}\cdot u'- \ell(u)^\beta\right) &= -2\beta \alpha \ell(u)^{\beta -1} |u'|^2\\\label{eq.derivative of H-v2}
	&= -\alpha\frac{2\beta |u'|^2}{(2\beta -1)|u'|^2 -B(u)} H(u)\le 0
\end{align}
\medskip
We will need the following two observations on the algebraic properties of $H$:
\begin{enumerate}[label=(H.\arabic*),ref=(H.\arabic*)]
	\item\label{property.H1} $H(u)\le 0$ implies $|u'|^2 \le C(B_0 + \sup |V|)$
	 and so $H(u)\ge -C (B_0 + \sup|V|)^\beta$
	\item\label{property.H2} $\frac{\partial H(u)}{\partial |u'|^2} \ge 0$ and so $\{V=0\} \subset \{ H \ge - B_0^{\beta} \}$
\end{enumerate} 
For the first we note that $H(u)<0$ implies that $(2\beta -1) |u'|^2 \le B(u)$. 
For the second a direct calculation reveals 
\[ \frac{\partial H}{\partial |u'|^2} = \ell(u)^{\beta -2} \left( (2\beta -1) \beta |u'|^2 + \beta B(u)\right) \ge 0 \,. \]
Hence $H(u,u') \ge H(u,0) = - B(u)^\beta$ and if $V(u)=0$ we conclude $H(u)\ge - B_0^\beta$. 

\subsubsection*{to Step 2:}\label{subsubsec.toStep2-equivariant}
``$\Leftarrow$'' follows from \ref{property.H1} since $H(u)<0$ implies $|u'|^2 \le C(B_0 +1)$ and so $\ell(u)^\beta \le C(B_0 +1)^\beta$ and therefore $E(u)\le \frac{C}{\alpha} (B_0+1)^\beta$. 

``$\Rightarrow$'' Suppose there is $t_0<0$ with $H(u(t_0))>0$, and so the monotonicity of $H$ gives $H(u(t))>0$ for $t<t_0$. Note \eqref{eq.derivative of H-v2} implies, in the case of $H(u)\ge 0$ and using the fact that $B(u)>0$, that  \[\frac{d}{dt} H \le - \alpha \frac{2\beta}{2\beta -1}H(u)\,.\]
Thus for any $t<t_0$ with $c=\alpha\frac{2\beta}{2\beta -1}$ the monotonicity of $H$ can be strengthened to
\[e^{ct_0}H(t_0) \le e^{ct} H(t)\,.\]
Hence there is $t_1<t_0$ with $H(u(t_1))\ge CB_0^\beta$. This implies firstly, that $H(u(t))\ge H(u(t_1))\ge C B_0$, secondly that $|u'|^2\ge 2(B_0+\sup|V|)$. Since $H(u) \le \max\{2\beta-1,\} \ell(u)^\beta$ we found that using the fact that  $c>\alpha$
\[ \int_{\tau}^{t_1} \ell(u)^\beta e^{\alpha t} \, dt \gtrsim \int_{\tau}^{t_1} H(u)  e^{\alpha t} \, dt \gtrsim C B_0^\beta \int_\tau^{t_1} e^{\alpha t - ct}\,dt \to +\infty \text{ as } \tau \downarrow -\infty\,.\]
This contradicts the finite energy assumption. 
\subsubsection*{to Step 3:}\label{subsubsec.toStep3-equivariant}
Let us introduce the energy restricted to the interval $(a,b)\subset \R$
\[ E_{(a,b)}(u)= \int_{a}^b \ell(u)^\beta\, e^{\alpha t} \,dt\,.\]
Due to step 2 we must have $H(u)\le 0$. Together with property \ref{property.H1}, it implies that \begin{equation}\label{eq.sup bound} |u(t)| \lesssim (B_0^{\nicefrac12} |t| +1)\,.\end{equation}
Given a sequence $t_k \downarrow -\infty$, we denote with $v_k$ a minimizer of 
\[\min\left\{ E_{(t_k,0)}(v) \colon v(0)=u(0) \right\}\,.\]
Then $v_k$ is characterised by solving \eqref{eq.Euler-Lagrange} with boundary conditions $v_k(0)=u(0)$ and $v_k'(t_k)=0$. First, ODE theory implies that $v_k$ is the unique minimizer. Second, the latter condition implies that $H(v_k(t_k)) < 0$. Appealing to the monotonicity of $H$ along $v$ we have $H(v_k(t)) \le 0$ for all $t_k \le t\le 0$ and so by \ref{property.H1} $v_k$ satisfies \eqref{eq.sup bound} as well. Let $w_k$ be the linear interpolation between $u(t_k)$ and $v_k(t_k+1)$ on $[t_k, t_k+1]$ given by \[w_k(t)=(t_k+1-t) u(t_k) + (t+t_k) v_k(t_k+1)\,.\]
It satisfies $|w_k'(t)| =|u(t_k)-v_k(t_k+1)|\lesssim B_0^{\nicefrac12}|t_k| +1$.
We construct a competitor for $u$ on $[0,t_k]$ using $w_k$ by
\[ \tilde{v}(t)=\begin{cases}
	v_k(t) &\text{ for } t_k+1 \le t \le 0\\
	w_k(t) &\text{ for } t_k \le t \le t_k\,.
\end{cases}\]
Since $u$ is the unique minimizer to $\min\{E_{(t_k,0)}(v)\colon v=u \text{ on } \partial [t_k,0]\}$ by ODE theory, we have 
\begin{align*}
	E_{(t_k,0)}(u) \le E_{(t_k,0)}(\tilde{v})&\le E_{(t_k+1,0)}(v_k)+ C \int_{t_k}^{t_k+1} \left(B_0|t_k|^2 +1\right)^\beta e^{\alpha t}\,dt\\
	&\le E_{(t_k,0)}(v_k)+ C \ \left(B_0|t_k|^2 +1\right)^\beta e^{\alpha t_k}\,.
\end{align*}
Hence we found 
\[ E_{(t_k,0)}(v_k)\le E_{(t_k,0)}(u)\le E_{(t_k,0)}(v_k) + C|t_k|^{2\beta}e^{\alpha t_k}\,.\]
Since $\lim_{k \to \infty} E_{(t_k,0)}(v_k)=\min\{ E(v)\colon v(0)=u(0)\}$ we can deduce, that $u$ is a minimizer. 

\subsubsection*{to Step 4:}\label{subsubsec.toStep4-equivariant}
Since solutions to \eqref{eq.Euler-Lagrange} exists for all time $u$ is defined on $\R$. Hence we can apply step 3 to $t \mapsto u(t+\delta)$ for any $\delta>0$ to deduce that $u$ is minimizing as well on $(-\infty, \delta)$. Now let $\tilde{u}$ be a minimizer of $\min\{ E_{(-\infty,0)}(v)\colon v(0)=u(0)\}$ then we obtain a competitor $\hat{u}$ to $u$ on $(-\infty,\delta)$ by extending it by $u$ i.e. 
\[ \hat{u}(t)=\begin{cases}
	\tilde{u}(t) &\text{ for }  t \le 0\\
	u(t) &\text{ for } 0 \le t \le \delta \,.
\end{cases}\]
Since $E_{(-\infty,0)}(\tilde{u})=E_{(-\infty,0)}(u)$, $\hat{u}$ is a minimizer for the interval $(-\infty,\delta)$ since $E_{(-\infty,\delta)}(\hat{u})=E_{(-\infty,\delta)}(u)$. But this implies that $\hat{u}$ satisfies \eqref{eq.Euler-Lagrange} and so due to uniqueness of ODE's we must have $\tilde{u}=u$.
\subsubsection*{to Step 5:}\label{subsubsec.toStep5-equivariant}
\begin{proof}[Proof of Lemma \ref{lem.H constant}]
If $H$ is constant then $u'=0$ due to \eqref{eq.derivative of H}. So \eqref{eq.Euler-Lagrange} implies that the LHS is vanishing and we must have $\frac{\partial V}{\partial u}=0$, but since $\{ \frac{\partial V}{\partial u}=0 \}= T^1 \times \{0\}$ we conclude that \eqref{eq.H constant: equivalence a} implies \eqref{eq.H constant: equivalence b}. The other direction is trivial. 
Now suppose that $H(u(0))<-B_0^\beta$. As we observed in Step 1 the structure of the ODE implies that solutions exists for all times. Now suppose that $u(t)$ remains in a bounded region of $T^1 \times \R$. Hence for any $t_i \to \infty$ the sequence $(u(t_i),u'(t_i))$ is bounded due to the algebraic property \ref{property.H1} of $H$. Thus there is a subsequence, not relabelled, that convergences. 
The smooth dependents on the initial conditions of solutions to ODE implies that the solutions $u_i(t)=u(t_i+t)$ converges locally smoothly to a solutions $v(t)$ of \eqref{eq.Euler-Lagrange}. 
But the monotonicity of $H$ along a solutions implies
\begin{equation}\label{eq.H < - B_0^beta}
	H(v(t))=\lim_{i \to \infty} H(u(t_i + t)) = \lim_{t\to \infty} H(u(t)) < - B_0^\beta\,.
\end{equation}
Hence $H(v(t))$ is constant, so that by the first part we have $v(t)\in T^1\times \{0\} \cap \{ H=-B_0^\beta \}$ is constant but this contradicts \eqref{eq.H < - B_0^beta}.
\end{proof}

\subsubsection*{to Step 6:}\label{subsubsec.toStep6-equivariant}
\begin{proof}[Proof of Lemma \ref{lem.non-unique tangent}]
	We consider initial data $u_n(0)=(\frac{\pi}{2},\frac{1}{2n\pi})$ with $u_n'(0)=0$. Since $V(u_n(0))<0$ one has
	\[-B_0^\beta - O\left(\frac1n\right)< H(u_n(0)) < -B_0^\beta\,,\]
	and so by the previous Lemma \ref{lem.H constant} the solution exists for all $t>0$ and must become unbounded, \eqref{eq.unbounded} . In fact $u_n(t)\in T^1\times (0,\infty)$ since by \eqref{eq.no intersection} the solution never crosses $T^1 \times \{0\}$. Following the original argument, let $t_n>0$ be the first time at which $u_n(t_n)\in T^1\times\{1\}$. Since $(u_n(0),u'_n(0))\to \bigl((\frac{\pi}{2},0),0\bigr)$ and the smooth dependents on the initial data for solutions of ODE's implies that $u_n(t)\to v_{const}(t)\equiv (\frac{\pi}{2},0)$ locally uniformly. Thus we must have $\liminf_n t_n= \infty$. Since $\tilde{u}_n(t)=u_n(t_n+t)$ satisfies $\tilde{u}_n(t)\in T^1\times (0,1]$ and $|\tilde{u}'_n(t)|^2\le C$ for all $[-t_n,0]$ there is a subsequence not relabelled such that $\tilde{u}_n(0)\to (\theta,1)$ for some $\theta \in T^1$ and $\lim_{n \to \infty}\tilde{u}'_n(0)$ exists. So that for that subsequence there is a solution $v(t)$ of \eqref{eq.Euler-Lagrange} with $\lim_{n\to \infty} \tilde{u}_n(t)=v(t)$ locally uniformly. 
	Since $H(\tilde{u}_n(t))< - B_0^\beta$ for all $t \ge -t_n$ we have that $H(v(t))\le -B_0^\beta$ for all $t\in \R$. 
	\medskip
	
	In fact the claimed conclusion \eqref{lem.non-unique tangent conclusion1} must hold.  Assume by contradiction that $H(v_0(t_0))=-B_0^\beta$. But then the monotonicity of $H$ along $v$ implies that $H(v(t))=-B_0^\beta$ for all $t\le t_0$ and so $v(t)\equiv (\theta, 0)$ by Lemma \ref{lem.H constant}, \eqref{eq.H constant: equivalence a}. Thus contradicting $v(0) \in T^1\times \{1\}$.
	To show the second part we argue analogously as in the proof of Lemma \ref{lem.H constant} part (2): Due to property \ref{property.H1} and $v(t) \in T^1 \times [0,1]$ for $t<0$ any sequence $t_n \downarrow -\infty$ contains a subsequence, not relabelled, such that $v(t_n)$ and $v'(t_n)$ converge. Hence $\lim_{n \to \infty} v(t_n + t)= w(t)$ locally uniformly. Additionally we have 
	\[H(w(t)) = \lim_{n \to \infty} H(v(t_n+ t)) = \lim_{t \to -\infty} H(v(t)) \le - B_0^\beta, \]
	 where we have used once again the monotonicity of $H(v(t))$. 
	 So that $H(w(t))=-B_0^\beta$ and $w(t)\equiv \lim_{n \to \infty} v(t_n) \in T^1 \times \{0\}$ constant. 
	 
	 Conclusion \eqref{lem.non-unique tangent conclusion2} follows from $H(v(t))< - B_0^\beta$, because this implies $|v'|^2\lesssim C B_0$ and so $\ell(v)^\beta \lesssim B_0^\beta$. Thus one found $E(u) \lesssim e^{\alpha R} \; B_0^\beta$.
	 
	 It remains to show that each given $(\theta_0, 0) \in T^1\times \{0\}$ is a possible limit. Since $v(t)$ is a continuous map into $T^1\times \R$ we can consider the lifted path $(\theta(t),y(t))\in \R\times \R$, where we implicitly used the identification $T^1=\nicefrac{ \R}{2\pi \Z}$. For a any fixed $t_0$ there is $k \in \Z$ such that  
	 \begin{equation}\label{eq.strip}
	 	k \pi< \theta(t_0) + \nicefrac{1}{y(t_0)} < (k+1) \pi\,.
	 \end{equation}
	 As observed in Lemma \ref{lem.H constant} conclusion \eqref{eq.no intersection} the solution $v(t)$ never enters 
	 \[\{ V=0\}=T^1\times \{0\} \cup \{ x+ \nicefrac1y = \pi \Z \}\,.\]
	 Hence the constraint \eqref{eq.strip} holds for all $t\in \R$. 
	 As we have just confirmed $y(t)\downarrow 0$ as $t \to -\infty$  hence $\theta(t)\downarrow - \infty$. Therefore the continuous curve $\theta(t)$ must cross $\theta_0 + 2\pi \Z$ infinitely many times. Thus there is a desired sequence $t_n \to - \infty$ with $\theta(t_n)-\theta_0 \in 2\pi \Z$ proving the final conclusion \eqref{lem.non-unique tangent conclusion3}.
\end{proof}

\bibliographystyle{plain}
\bibliography{lit}
\end{document}